\documentclass[11pt]{amsart}  
\usepackage{amsmath,amssymb,amsthm,amsfonts}
\usepackage{mathrsfs}
\usepackage{xcolor}
\usepackage[all,cmtip]{xy}

\usepackage{geometry}
 \usepackage[dvipsnames]{xcolor} 
\usepackage[colorlinks]{hyperref}

\usepackage[color=Orange!50!white, textsize=scriptsize]{todonotes}

\usepackage[noabbrev,capitalize]{cleveref}
\crefname{equation}{}{}

\usepackage{cancel}

\usepackage{thmtools} 
\usepackage{thm-restate}
\usepackage{lipsum}

\newtheorem{theorem}{Theorem}[section]
\newtheorem{proposition}[theorem]{Proposition}

\newtheorem{lemma}[theorem]{Lemma}

\newtheorem{corollary}[theorem]{Corollary}

\newtheorem*{question*}{Question}

\theoremstyle{definition}

\newtheorem{defn}[theorem]{Definition}

\newtheorem*{definition*}{Definition}

\theoremstyle{remark}

\newcommand{\EE}{\mathbb{E}}
\newcommand{\CC}{\mathbb{C}}

\newcommand{\RR}{\mathbb{R}}

\newcommand{\FF}{\mathbb{F}}
\newcommand{\TT}{\mathbb{T}}
\newcommand{\ZZ}{\mathbb{Z}}

\newcommand{\cP}{\mathcal{P}}

\title{Simultaneous popular polynomial differences over finite fields}

\author{David Conlon}
\address{Department of Mathematics, California Institute of Technology, Pasadena, CA 91125, USA}
\email{\{dconlon,ddong124,ghong\}@caltech.edu}
\thanks{Research supported by NSF Award DMS-2348859.}

\author{Dingding Dong}

\author{Guo-Dong Hong}

\begin{document}

\maketitle

\begin{abstract}
Green's popular difference theorem says that for every \(\varepsilon>0\), all sufficiently large primes \(p\), and every set \(A\subseteq\mathbb F_p\) of density \(\alpha\), there exists a nonzero \(d\in\mathbb F_p\) such that 
\[
    \mathbb E_{x\in\mathbb F_p}
    1_A(x)1_A(x+d)1_A(x+2d)
    \geq
    \alpha^3-\varepsilon.
\]
We show that a stronger simultaneous popular difference phenomenon holds for polynomial configurations. Namely, if $\mathcal P=\{P_1,\dots,P_k\} \subset \mathbb Z[t]$ is a fixed collection of linearly independent polynomials with zero constant terms, we show that for every \(\varepsilon>0\), all sufficiently large primes \(p\), and every set \(A\subseteq\mathbb F_p\) of density \(\alpha\), there exists a nonzero \(d\in\mathbb F_p\) such that
\[
    \mathbb E_{x\in\mathbb F_p}
    1_A(x)
    \prod_{i=1}^k
    1_A\bigl(x+P_i(d)\bigr)^{\omega_i}
    \geq
    \alpha^{1+\sum_i\omega_i}-\varepsilon
\]
simultaneously for every \(\omega=(\omega_1,\dots,\omega_k)\in\{0,1\}^k\). 

We also show that such simultaneous popular difference phenomena have sharp limitations by proving that for every sufficiently large prime \(p\), there is a constant \(c>0\) such that, for all sufficiently large \(n\), one can find a set \(A\subseteq\mathbb F_p^n\) of density \(1/2+o_n(1)\) satisfying
\[
    \max_{d\neq 0}
    \min\left\{
        \mathbb E_{x\in\mathbb F_p^n}
        1_A(x)1_A(x+d)1_A(x+2d),
        \mathbb E_{x\in\mathbb F_p^n}
        1_A(x)1_A(x+2d)1_A(x+4d)
    \right\}
    \leq
    \frac18-c.
\]
That is, the strengthening of Green's result, in this case over $\mathbb F_p^n$ for $p$ fixed and $n$ tending to infinity, requiring that both \(d\) and \(2d\) are simultaneously popular differences for three-term arithmetic progressions is false.  
\end{abstract}

\section{Introduction}

Szemerédi's theorem \cite{S75} asserts that every subset of the integers with positive upper density contains arbitrarily long arithmetic progressions.  Furstenberg \cite{F77} gave a proof of Szemerédi's theorem using ergodic theory, showing that the combinatorial statement follows from a multiple recurrence theorem for measure-preserving systems.  This ergodic viewpoint has since become of fundamental importance in additive combinatorics and Ramsey theory.

Using this viewpoint, a far-reaching polynomial extension of Szemerédi's theorem was obtained by Bergelson and Leibman \cite{BL96}. Their polynomial Szemerédi theorem, as it is known, states that if \(A\subseteq \mathbb Z\) has positive upper density and \(P_1,\dots,P_k \in \mathbb Z[t]\) satisfy \(P_i(0)=0\) for each $i = 1, \dots, k$, then \(A\) contains polynomial configurations of the form
\[
    x, x+P_1(d), \dots, x+P_k(d)
\]
with \(d\neq 0\).  Equivalently, in ergodic terms, if \((X,\mathcal B,\mu,T)\) is an invertible measure-preserving system and \(A\in\mathcal B\) has positive measure, then one has polynomial multiple recurrence, meaning that 
\[
    \mu\bigl(A\cap T^{-P_1(d)}A\cap\cdots\cap T^{-P_k(d)}A\bigr)>0
\]
for some \(d\in\mathbb N\).

One can ask for stronger recurrence statements in which the relevant intersection is not merely positive, but has essentially the size predicted by random behavior.  In the single-recurrence setting, Khintchine's recurrence theorem states that, for every \(\varepsilon>0\), there exists \(d\in\mathbb N\) such that
\[
    \mu(A\cap T^{-d}A)>\mu(A)^2-\varepsilon.
\]
Much work has gone into extending this result in various directions. For instance, a result of Frantzikinakis and Kra \cite{FK06} states that if \(P_1,\dots,P_k\) is a family of linearly independent polynomials, then, for every \(\varepsilon>0\), there exists \(d\in\mathbb N\) such that
\[
    \mu(A\cap T^{-P_1(d)}A \cap \dots \cap T^{-P_k(d)}A)>\mu(A)^{k+1}-\varepsilon.
\]

A different strengthening of Khintchine's theorem, which they called simultaneous polynomial recurrence, was found by Lyall and Magyar \cite{LM11}. 
They showed that, for any family of polynomials \(P_1,\dots,P_k\) with $P_i(0) = 0$ for all $i = 1, \dots, k$ and every \(\varepsilon>0\), there exists \(d\in\mathbb N\) such that 
\[
    \mu(A\cap T^{-P_i(d)}A) > \mu(A)^2-\varepsilon
\]
for all $i = 1, \dots, k$. That is, a Khintchine-type lower bound holds simultaneously for each polynomial recurrence. 

The first purpose of this paper is to establish a strengthened analogue of this simultaneous recurrence phenomenon, a common generalization of the Frantzikinakis--Kra and Lyall--Magyar results, over finite fields. We consider dense subsets of \(\mathbb F_p\) and polynomial configurations
\[
    x, x+P_1(d), \dots, x+P_k(d),
\]
where \(d\in\mathbb F_p\setminus\{0\}\). Provided that the polynomials are linearly independent, we prove that there exists a single nonzero parameter \(d\) for which all subconfigurations of this polynomial pattern occur with essentially random density.

More formally, let $\mathcal P=\{P_1,\dots,P_k\} \subset \mathbb Z[t]$ be a collection of polynomials.  For \(d\neq 0\), define
\[
    I_{\mathcal P}(f_0,\dots,f_k)(d)
    :=
    \mathbb E_{x\in\mathbb F_p}
    f_0(x)\prod_{i=1}^k f_i(x+P_i(d)).
\]
Given $\omega=(\omega_1,\dots,\omega_k)\in\{0,1\}^k$, we define \(\mathcal P_\omega\) to be the subcollection of \(\mathcal P\) containing exactly those \(P_i\) with \(\omega_i=1\) and set 
\[
    I_{\mathcal P_\omega}(f_0,\dots,f_k)(d)
    :=
    \mathbb E_{x\in\mathbb F_p}
    f_0(x)
    \prod_{i=1}^k
    f_i\bigl(x+P_i(d)\bigr)^{\omega_i}.
\]
With this notation, our first result is as follows. 

\begin{theorem}
\label{thm_simultaneous polynomial common difference}
Suppose \(k\geq 1\) and $\mathcal P=\{P_1,\dots,P_k\} \subset \mathbb Z[t]$ is a collection of linearly independent polynomials with zero constant terms.  Then, for every \(\varepsilon>0\), there exists $p_0=p_0(\varepsilon,\mathcal P)$ such that, for every prime \(p\geq p_0\) and every \(A\subseteq\mathbb F_p\), there exists \(d\in\mathbb F_p\setminus\{0\}\) such that
\[
    I_{\mathcal P_\omega}(1_A,\dots,1_A)(d)
    \geq
    \alpha^{\,1+\sum_i\omega_i}-\varepsilon
\]
holds simultaneously for every $\omega=(\omega_1,\dots,\omega_k)\in\{0,1\}^k$,
where $\alpha=\mathbb E_{x\in\mathbb F_p}1_A(x)$.
\end{theorem}

Thus, the same parameter \(d\) works not only for the full configuration, but for every subconfiguration obtained by selecting an arbitrary subset of the polynomial shifts \(P_i(d)\).  In particular, taking \(\omega_i=1\) for all \(i\) gives
\[
    I_{\mathcal P}(1_A,\dots,1_A)(d)
    \geq
    \alpha^{k+1}-\varepsilon,
\]
while taking \(\omega\) supported on a single index gives the Khintchine-type lower bound
\[
    \mathbb E_{x\in\mathbb F_p}
    1_A(x)1_A(x+P_i(d))
    \geq
    \alpha^2-\varepsilon
\]
simultaneously for all \(1\leq i\leq k\).

The second purpose of this paper is to show that simultaneous recurrence phenomena of this type have sharp limitations.  We focus on the popular common difference problem for arithmetic progressions over the group $\mathbb{F}_p^n$, where we think of $p$ as being fixed and $n$ tending to infinity.  Given \(f:\mathbb F_p^n\to\mathbb R\), define, for \(d\in\mathbb F_p^n\setminus\{0\}\),
\[
    I_3(f)(d)
    :=
    \mathbb E_{x\in\mathbb F_p^n}
    f(x)f(x+d)f(x+2d).
\]
When \(f=1_A\), this quantity counts the density of three-term arithmetic progressions in \(A\) with common difference \(d\). The popular common difference problem asks whether, for every \( \varepsilon>0\) and for $n$ sufficiently large in terms of  $\varepsilon$, one can find, in every $A \subseteq \mathbb{F}_p^n$ with $\alpha:=\mathbb E 1_A$, a nonzero \(d\) for which the density of three-term progressions in $A$ with common difference \(d\) is within $\varepsilon$ of the random lower bound, that is, 
\[
    I_3(1_A)(d)
    \geq
    \alpha^3-\varepsilon.
\]
This problem was raised by Bergelson, Host, and Kra \cite{BHK05}, who proved weaker related results using ergodic methods, and solved positively, over every finite abelian group, by Green~\cite{G05} using his arithmetic regularity lemma. His result was later extended to four-term progressions by Green and Tao~\cite{GT10}, while an example due to Ruzsa, given in an appendix to~\cite{BHK05}, shows that there is no similar phenomenon for longer progressions. 

Our Theorem~\ref{thm_simultaneous polynomial common difference} may be viewed as a variant of Green's result, showing that over $\mathbb{F}_p$ one can find popular differences for longer polynomial progressions, even simultaneously for a collection of distinct progressions, provided the polynomials are linearly independent. This raises the question of whether some simultaneous popular difference phenomenon might be possible for arithmetic progressions. For instance, can one always find a nonzero $d$ for which both $d$ and $2d$ are popular differences for three-term progressions? Our second result shows that this is not the case over $\mathbb{F}_p^n$ for $p$ fixed and $n$ tending to infinity.

\begin{theorem}
\label{thm:k>2}
For any sufficiently large prime $p$, there exists an absolute constant $c > 0$ such that, for all sufficiently large $n$, there exists a set $A\subseteq\mathbb F_p^n$ with density $\frac12+o_n(1)$ such that
\[
    \max_{d\neq 0}
    \min\bigl\{
        I_3(1_A)(d),
        I_3(1_A)(2d)
    \bigr\}
    \leq
    \frac18-c.
\]
\end{theorem}


That is, for such an $A$, at least one of the two common differences $d$ and $2d$ fails to be popular for every nonzero $d$. For any other fixed natural number $k \ge 3$, a similar construction yields examples where $d$ and $kd$ are not both popular common differences for any nonzero $d$. 

The rest of the paper is organized as follows.  In Section~\ref{sec:preliminaries}, we fix some notation and conventions and record some basic facts that will be used throughout.  In Sections~\ref{sec:positive-main-proof} and~\ref{sec:negative-ap}, we prove Theorems~\ref{thm_simultaneous polynomial common difference} and~\ref{thm:k>2}, respectively. We conclude with some further remarks and open problems. 

\section{Preliminaries} \label{sec:preliminaries}

\subsection{Notation and conventions}

Throughout the paper, \(p\) will denote an odd prime.  We write \(\mathbb F_p\) for the finite field with \(p\) elements and \(\mathbb F_p^n\) for the \(n\)-dimensional vector space over \(\mathbb F_p\).  We identify \(\mathbb F_p\) with the residue classes
\[
    \{0,1,\dots,p-1\}
\]
when discussing representatives modulo \(p\).

For a finite set \(X\), we write
\[
    \mathbb E_{x\in X} f(x)
    :=
    \frac{1}{|X|}\sum_{x\in X} f(x).
\]
When the underlying set is clear, we simply write \(\mathbb E_x f(x)\). 

We write
\[
    e_p(t):=\exp(2\pi i t/p)
\]
for the standard additive character on \(\mathbb F_p\).  In \(\mathbb F_p^n\), we use the dot product
\[
    \langle x,\xi\rangle
    :=
    x_1\xi_1+\cdots+x_n\xi_n,
    \qquad
    x,\xi\in\mathbb F_p^n.
\]
For \(f:\mathbb F_p^n\to\mathbb C\), we then have the Fourier transform
\[
    \widehat f(\xi)
    :=
    \mathbb E_{x\in\mathbb F_p^n}
    f(x)e_p(-\langle x,\xi\rangle),
    \qquad
    \xi\in\mathbb F_p^n.
\]
In the special case where \(n=1\), this becomes
\[
    \widehat f(\xi)
    =
    \mathbb E_{x\in\mathbb F_p}
    f(x)e_p(-\xi x).
\]
With this normalization, Fourier inversion is given by
\[
    f(x)
    =
    \sum_{\xi\in\mathbb F_p^n}
    \widehat f(\xi)e_p(\langle x,\xi\rangle)
\]
and Plancherel's identity is
\[
    \mathbb E_{x\in\mathbb F_p^n}|f(x)|^2
    =
    \sum_{\xi\in\mathbb F_p^n}|\widehat f(\xi)|^2.
\]

For \(1\le q<\infty\), we define
\[
    \|f\|_{L^q(\mathbb F_p^n)}
    :=
    \left(
        \mathbb E_{x\in\mathbb F_p^n}|f(x)|^q
    \right)^{1/q},
\]
with the usual interpretation when \(q=\infty\).  For a function
\(F:\mathbb F_p^n\to\mathbb C\) on the frequency side, we write
\[
    \|F\|_{\ell^q(\mathbb F_p^n)}
    :=
    \left(
        \sum_{\xi\in\mathbb F_p^n}|F(\xi)|^q
    \right)^{1/q}.
\]
In particular,
\[
    \|\widehat f\|_{\ell^\infty(\mathbb F_p^n)}
    =
    \max_{\xi\in\mathbb F_p^n}|\widehat f(\xi)|.
\]

For functions \(F,G:\mathbb F_p^n\to\mathbb C\) on the frequency side, we use the unnormalized convolution
\[
    (F*G)(\xi)
    :=
    \sum_{\eta\in\mathbb F_p^n}F(\eta)G(\xi-\eta).
\]
With our Fourier normalization, one has
\[
    \widehat{fg}
    =
    \widehat f*\widehat g.
\]
Consequently,
\[
    \|\widehat{fg}\|_{\ell^1}
    \leq
    \|\widehat f\|_{\ell^1}\|\widehat g\|_{\ell^1}.
\]

We say that \(f\) is \(1\)-bounded if
\[
    |f(x)|\leq 1
\]
for every $x$. 

We write \(1_A\) for the indicator function of a set \(A\).  The density of a set \(A\subseteq G\), where \(G\) is a finite abelian group, is
\[
    \alpha
    :=
    \mathbb E_{x\in G}1_A(x)
    =
    \frac{|A|}{|G|}.
\]

For $x,h \in \FF_p^n$, we define the multiplicative derivative by
\[
\triangle_{h}f(x):=f(x)\overline{f}(x+h).
\]

For \(t\in\mathbb R/\mathbb Z \cong \TT\), we write
\[
    \|t\|_{\TT}
\]
for the distance from \(t\) to the nearest integer.  

Finally, we use the notation \(X\lesssim Y\) to mean that \(X\leq CY\) for an absolute constant \(C>0\).  If the implicit constant depends on additional parameters, we indicate this by subscripts; for example,
\[
    X\lesssim_{\mathcal P,r}Y
\]
means that the implicit constant may depend on \(\mathcal P\) and \(r\), but not on \(p\) or \(n\).  Similarly, \(O_{\mathcal P}(Y)\) denotes a quantity bounded in magnitude by \(C_{\mathcal P}Y\).

\subsection{Weil bound for polynomial exponential sums}

We will use the following standard form of the Weil bound for additive character sums over finite fields (see, e.g., \cite[Section 3]{Kow}).  It says that a nonconstant polynomial phase exhibits square-root cancellation provided that the degree is smaller than the characteristic.  In our applications, the degree will be fixed while \(p\) tends to infinity.

\begin{lemma}[Weil bound] \label{WBound}
Let \(f\in\mathbb F_p[x]\) be a polynomial of degree \(d\) with $0<d<p$. Then
\[
    \left|
        \mathbb E_{x\in\mathbb F_p} e_p(f(x))
    \right|
    \lesssim_d
    p^{-1/2}.
\]
\end{lemma}

\subsection{Bohr sets}

Though \(\mathbb F_p\) has no nontrivial proper subgroups, there are subsets, known as Bohr sets, which may stand in as approximate subgroups. These Bohr sets play an important role in additive combinatorics and we will make use of them here. 

\begin{defn}[Bohr sets]
Let \(\Gamma\subseteq \mathbb F_p\setminus\{0\}\) be a set of frequencies and let \(\rho\in(0,1/2]\).  Define the 
Bohr set $B=B(\Gamma,\rho)$ by 
\[
    B(\Gamma,\rho)
    :=
    \Bigl\{
        x\in\mathbb F_p:
        \bigl\|\tfrac{\xi x}{p}\bigr\|_{\TT}\leq \rho
        \text{ for every }\xi\in\Gamma
    \Bigr\}.
\]
The rank of \(B\) is then \(|\Gamma|\).
\end{defn}

We record two useful estimates regarding Bohr sets.

\begin{lemma}[Lemma 4.20 in \cite{TV06}] \label{lma_B_size}
Let \(B=B(\Gamma,\rho)\subseteq\mathbb F_p\) be a Bohr set.  Then
\[
    |B|\geq \rho^{|\Gamma|}|\mathbb F_p|.
\]
\end{lemma}

\begin{lemma} \label{lma:l^1 for Bohr}
Let \(B=B(\Gamma,\rho)\) be a Bohr set of rank \(r>0\).  Then
\[
    \sum_{\xi\in\mathbb F_p}
    \bigl|\widehat{1_B}(\xi)\bigr|
    \lesssim_r
    (\log p)^r.
\]
\end{lemma}

\begin{proof}
For \(\rho\in(0,1/2]\), set
\[
    I_\rho
    :=
    \Bigl\{
        t\in\mathbb F_p:
        \bigl\|\tfrac{t}{p}\bigr\|_{\mathbb T}\leq \rho
    \Bigr\}.
\]
Identifying \(\mathbb F_p\) with \(\{0,1,\dots,p-1\}\), the condition
\[
    \left\|\frac{t}{p}\right\|_{\mathbb T}\leq \rho
\]
is equivalent to \(t\equiv s \pmod p\) for some integer \(s\) with \(|s|\leq M\), where
\[
    M:=\lfloor \rho p\rfloor.
\]
Thus,
\[
    I_\rho=\{-M,-M+1,\dots,M\}\pmod p,
\]
a symmetric interval modulo \(p\).

We first prove the desired \(\ell^1\)-Fourier bound for intervals.  Let
\[
    I=\{0,1,\dots,m-1\}\subseteq\mathbb F_p
\]
be an interval of length \(m = 2M+1\).  The symmetric interval case is equivalent up to translation, which does not affect the absolute values of the Fourier coefficients.  For \(\xi\neq 0\), we compute
\[
    \widehat{1_I}(\xi)
    =
    \mathbb E_x 1_I(x)e_p(-\xi x)
    =
    \frac1p\sum_{x=0}^{m-1}e_p(-\xi x)
    =
    \frac1p\cdot
    \frac{1-e_p(-\xi m)}{1-e_p(-\xi)}.
\]
Using $|1-e_p(-\xi m)|\leq 2$ and the standard bound 
\[
    |1-e_p(-\xi)|
    \asymp
    \frac{|\xi|}{p}
\]
for $1\leq |\xi|\leq (p-1)/2$, where \(|\xi|\) denotes the least absolute residue of \(\xi\), we obtain
\[
    |\widehat{1_I}(\xi)|
    \lesssim
    \frac1p
    \min\left(m,\frac{p}{|\xi|}\right)
    =
    \min\left(\frac{m}{p},\frac1{|\xi|}\right).
\]
Therefore,
\[
    \sum_{\xi\in\mathbb F_p}|\widehat{1_I}(\xi)|
    \leq
    |\widehat{1_I}(0)|
    +
    2\sum_{\xi=1}^{(p-1)/2}
    |\widehat{1_I}(\xi)|
    \lesssim
    1+\sum_{\xi=1}^{p/2}
    \min\left(\frac{m}{p},\frac1{\xi}\right).
\]
Let \(R_0:=\lfloor p/m\rfloor\).  Then
\[
    \sum_{\xi=1}^{p/2}
    \min\left(\frac{m}{p},\frac1{\xi}\right)
    \leq
    \sum_{\xi\leq R_0}\frac{m}{p}
    +
    \sum_{\xi>R_0}\frac1{\xi}
    \lesssim
    1+\log p.
\]
Hence, 
\begin{equation*}
    \|\widehat{1_I}\|_{\ell^1(\mathbb F_p)}
    \lesssim
    \log p.
\end{equation*}
The same bound then also applies to $\|\widehat{1_{I_\rho}}\|_{\ell^1(\mathbb F_p)}$. 

Next, for each \(\xi\in\Gamma\), we have
\[
    x\in B(\Gamma,\rho)
    \quad\Longleftrightarrow\quad
    \xi x\in I_\rho
    \text{ for every }\xi\in\Gamma.
\]
Hence,
\begin{equation}\label{eq:indicator-factor}
    1_B(x)
    =
    \prod_{\xi\in\Gamma}1_{I_\rho}(\xi x)
    =
    \prod_{\xi\in\Gamma}1_{\xi^{-1}I_\rho}(x).
\end{equation}
Since \(\widehat{fg}=\widehat f*\widehat g\), we have
\[
    \|\widehat{fg}\|_{\ell^1}
    \leq
    \|\widehat f\|_{\ell^1}
    \|\widehat g\|_{\ell^1}.
\]
Iterating this inequality and using \eqref{eq:indicator-factor}, we obtain
\[
    \|\widehat{1_B}\|_{\ell^1(\mathbb F_p)}
    \leq
    \prod_{\xi\in\Gamma}
    \bigl\|
        \widehat{1_{\xi^{-1}I_\rho}}
    \bigr\|_{\ell^1(\mathbb F_p)}.
\]
Multiplication by a nonzero scalar \(\xi^{-1}\) only permutes the frequency side, so
\[
    \bigl\|
        \widehat{1_{\xi^{-1}I_\rho}}
    \bigr\|_{\ell^1(\mathbb F_p)}
    =
    \|\widehat{1_{I_\rho}}\|_{\ell^1(\mathbb F_p)}.
\]
Therefore,
\[
    \|\widehat{1_B}\|_{\ell^1(\mathbb F_p)}
    \leq
    \left(
        \|\widehat{1_{I_\rho}}\|_{\ell^1(\mathbb F_p)}
    \right)^r
    \lesssim_r
    (\log p)^r,
\]
completing the proof.
\end{proof}

\section{Proof of Theorem 
\ref{thm_simultaneous polynomial common difference}} \label{sec:positive-main-proof}

In this section we prove Theorem~\ref{thm_simultaneous polynomial common difference}.  The proof has three main ingredients.  First, we use Peluse's finite-field polynomial Szemerédi theorem to obtain a \(U^2\)-type control estimate for polynomial counting operators.  Second, we apply an arithmetic regularity decomposition to split \(1_A\) into structured, pseudorandom, and small components.  Third, we use the Bohr-set estimates from Section~\ref{sec:preliminaries} to find many parameters \(d\) for which all polynomial shifts \(P_i(d)\) lie in the Bohr set controlling the structured component.

We begin by introducing an averaged polynomial counting operator associated to a collection $\mathcal P=\{P_1,\dots,P_k\} \subset \mathbb{Z}[t]$. 
For functions \(f_0,\dots,f_k:\mathbb F_p\to\mathbb C\), define
\[
\Lambda_{\mathcal{P}}( f_0, \dots,f_{k})=
    \EE_{x,d \in \FF_p}f_0(x) \prod_{i=1}^k f_i( x+ P_i(d)).
\] 

The following theorem of Peluse shows that these polynomial configurations satisfy an asymptotic independence property when the polynomials are linearly independent.

\begin{theorem}[Polynomial Szemerédi theorem, \cite{P19}] \label{thm_polynomial Szemeredi}
    Let $\mathcal{P}= \{P_1, \dots, P_{k}\} \subset \ZZ[t]$ be a collection of linearly independent polynomials with zero constant terms. Then there exists $C=C(\mathcal{P})>0$ such that
    \[
    \Lambda_{\mathcal{P}}( f_0, \dots,f_{k})= \prod_{i=0}^k \EE f_i +O_{\mathcal{P}}(p^{-C})
    \]
    for any 1-bounded functions $f_0, \dots, f_k $.
\end{theorem}

We next derive a consequence which will be used to control the pseudorandom part of the regularity decomposition.  The point is that if one of the functions has small Fourier coefficients, then the corresponding polynomial count is small for most values of the parameter \(d\).  This is the finite-field analogue of the usual generalized von Neumann principle, but here it follows directly from Theorem~\ref{thm_polynomial Szemeredi} after applying it to multiplicative derivatives.

\begin{lemma} \label{lma_U^2-control for poly}
    Let $\mathcal{P}= \{P_1, \dots, P_{k}\} \subset \ZZ[t]$ be a collection of linearly independent polynomials with zero constant terms. Then there exists $C=C(\mathcal{P})>0$ such that
    \[
    \EE_{d \in \FF_p} |I_{\cP}(f_0,\dots,f_k)(d)|^2 \leq
    \min_{i} \|\widehat{f_i}\|_{\ell^{\infty}(\FF_p)} + O_{\mathcal{P}}(p^{-C})
    \]
    for any 1-bounded functions $f_0, \dots, f_k $.
\end{lemma}

\begin{proof}
Expanding the square and introducing the new variable \(h=x'-x\), we get that
\begin{align*}
    \EE_{d \in \FF_p} |I_{\cP}(f_0,\dots,f_k)(d)|^2
    &=
    \EE_{d \in \FF_p}
    \left( \EE_{x \in \FF_p} f_0(x) \prod_{i=1}^k f_i( x+ P_i(d))\right)
    \left( \EE_{x' \in \FF_p} \bar{f_0}(x') \prod_{i=1}^k \bar{f_i}( x'+ P_i(d))\right)\\
    &= \EE_{h \in \FF_p} 
    \left(
    \EE_{x,d \in \FF_p} \triangle_h f_0(x)
    \prod_{i=1}^k \triangle_hf_i( x+ P_i(d))
    \right).
\end{align*}
For each fixed \(h\), the functions \(\triangle_h f_i\) are still \(1\)-bounded.  Applying Theorem~\ref{thm_polynomial Szemeredi} to these functions gives
\begin{align*}
    \EE_{d \in \FF_p} |I_{\cP}(f_0,\dots,f_k)(d)|^2
    &= \EE_{h \in \FF_p}
    \left(
        \prod_{i=0}^k
        \EE_{x \in \FF_p} \triangle_h f_i(x)
    \right)
    +O_{\mathcal{P}}(p^{-C}) \\
    &\leq \min_{i}
    \EE_{h \in \FF_p}
    \left|
        \EE_{x \in \FF_p} \triangle_h f_i(x)
    \right|
    +O_{\mathcal{P}}(p^{-C}),
\end{align*}
where we used the \(1\)-boundedness of the functions.  By Cauchy--Schwarz,
\begin{align*}
    \EE_{h \in \FF_p}
    \left|
        \EE_{x \in \FF_p} \triangle_h f_i(x)
    \right|
    \leq
    \left(
        \EE_{h \in \FF_p}
        \left|
            \EE_{x \in \FF_p} \triangle_h f_i(x)
        \right|^2
    \right)^{1/2}.
\end{align*}
Finally,
\[
\EE_{h \in \FF_p}
\left|
    \EE_{x \in \FF_p} \triangle_h f_i(x)
\right|^2
=
\|\widehat{f_i}\|_{\ell^4(\FF_p)}^4.
\]
Since \(f_i\) is \(1\)-bounded, \(\|\widehat{f_i}\|_{\ell^2}\leq 1\), and so
\[
    \|\widehat{f_i}\|_{\ell^4}^4
    \leq
    \|\widehat{f_i}\|_{\ell^\infty}^2
    \|\widehat{f_i}\|_{\ell^2}^2
    \leq
    \|\widehat{f_i}\|_{\ell^\infty}^2.
\]
Therefore,
\[
    \left(
        \EE_{h \in \FF_p}
        \left|
            \EE_{x \in \FF_p} \triangle_h f_i(x)
        \right|^2
    \right)^{1/2}
    \leq
    \|\widehat{f_i}\|_{\ell^\infty(\FF_p)}.
\]
Taking the minimum over \(i\) completes the proof.
\end{proof}

We now record the arithmetic regularity decomposition that will be used to split \(1_A\).  The structured component is approximately invariant under translations from a low-rank Bohr set, the pseudorandom component has small Fourier coefficients, and the small component is negligible in \(L^2\).

\begin{proposition}[Arithmetic regularity decomposition for $U^2$-norm, \cite{BSST22} and \cite{T14}] 
\label{prop_arithmetic regularity decomposition}
    Let $G$ be a compact abelian group with Haar probability measure $\mu$. Fix parameters $\delta, \epsilon>0$, growth functions $\eta_1, \eta_2: \RR_+ \rightarrow \RR_+$, and a finite set $\Gamma_0 \subset \widehat{G}$. Given any function $f : G \rightarrow[0,1]$, there exist $\rho_1, \rho_2>0$ and a finite set $\Gamma$ such that $\Gamma_0 \subseteq \Gamma \subseteq \widehat{G}$ with the following properties: 
    \begin{enumerate}
        \item $|\Gamma|=O_{\delta, \epsilon, |\Gamma_0|, \eta_1, \eta_2}(1)$.

        \item $\rho_1 \leq 1/ \eta_1(|\Gamma|+\delta^{-1}+\epsilon^{-1})$ and $\rho_2 \leq 1/ \eta_2(\rho_1^{-1})$, where $\rho_2$ is bounded away from zero independently of $f$.

        \item There is a decomposition $f=f_{\text{str}}+f_{\text{psr}}+f_{\text{sml}}$, where
        \begin{enumerate}
            \item $f_{\text{str}},f_{\text{psr}},f_{\text{sml}}$ are all 1-bounded,

            \item $f_{\text{str}}$ is nonnegative, has mean $\int f_{\text{str}} \, d \mu= \int f \, d \mu$, and obeys
            \[
            f_{\text{str}}(x+d)=f_{\text{str}}(x)+O(\epsilon)
            \]
            whenever $x \in G$ and $d \in B(\Gamma, \rho_1)$,

            \item $\|\widehat{f_{\text{psr}}} \|_{\ell^{\infty}(\widehat{G})} \leq \rho_2$,

            \item $\|f_{\text{sml}}\|_{L^2(G)} \leq \epsilon$.
        \end{enumerate}
    \end{enumerate}
\end{proposition}

The next lemma guarantees that the polynomial images of a random parameter \(d\) are equidistributed with respect to a fixed low-rank Bohr set.  This is another place where the linear independence of the polynomials is needed.

\begin{lemma} \label{lma_equdistribution for poly. sequence}
Let $\mathcal{P}= \{P_1, \dots, P_{k}\} \subset \ZZ[t]$ be a collection of linearly independent polynomials with zero constant terms and let $B=B(\Gamma,\rho)$ be a Bohr set with rank $r>0$. Then
\[
    \EE_{d\in \FF_p}  \prod_{i=1}^k \mathbf 1_B(P_i(d))=
    \left( \EE \mathbf 1_B \right)^{k}+ O_{\cP,r}\left(p^{-1/2} ( \log p)^{kr} \right).
\]

\end{lemma}

\begin{proof}
By Fourier inversion,
\begin{align*}
    \EE_{d\in \FF_p}  \prod_{i=1}^k \mathbf 1_B(P_i(d))
    &=
    \EE_{d\in \FF_p}
    \sum_{\xi_1, \dots, \xi_k \in \FF_p}
    \left(
        \prod_{i=1}^k \widehat{\mathbf 1_B}(\xi_i)
    \right)
    e_p \left( \sum_{i=1}^k \xi_i P_i(d) \right)\\
    &=
    \sum_{\xi_1, \dots, \xi_k \in \FF_p}
    \left(
        \prod_{i=1}^k \widehat{\mathbf 1_B}(\xi_i)
    \right) 
    \EE_{d\in \FF_p}
    e_p \left(  \sum_{i=1}^k \xi_i P_i(d) \right).
\end{align*}
The contribution of \((\xi_1,\dots,\xi_k)=(0,\dots,0)\) is
\[
    \left(\mathbb E 1_B\right)^k.
\]
For every nonzero tuple \((\xi_1,\dots,\xi_k)\), the polynomial
\[
    \sum_{i=1}^k \xi_i P_i(d)
\]
is nonconstant, since the polynomials \(P_1,\dots,P_k\) are linearly independent.  Hence, for all sufficiently large \(p\), Lemma~\ref{WBound} gives
\[
    \left|
    \EE_{d\in \FF_p}
    e_p\left(\sum_{i=1}^k \xi_i P_i(d)\right)
    \right|
    \lesssim_{\mathcal P}
    p^{-1/2}.
\]
Therefore, by Lemma~\ref{lma:l^1 for Bohr},
\begin{align*}
    &\left|
    \sum_{\substack{\xi_1,\dots,\xi_k\in\mathbb F_p\\
    (\xi_1,\dots,\xi_k)\neq (0,\dots,0)}}
    \left(
        \prod_{i=1}^k \widehat{\mathbf 1_B}(\xi_i)
    \right)
    \EE_{d\in \FF_p}
    e_p\left(\sum_{i=1}^k \xi_i P_i(d)\right)
    \right|\\
    &\qquad\lesssim_{\mathcal P}
    p^{-1/2}
    \prod_{i=1}^k
    \|\widehat{1_B}\|_{\ell^1(\FF_p)}
    \lesssim_{\mathcal P,r}
    p^{-1/2}(\log p)^{kr}.
\end{align*}
This proves the lemma.
\end{proof}

We are now ready to prove our main positive result, that the popular common difference phenomenon can be made to hold simultaneously for all subsets of a finite set of linearly independent polynomials $\mathcal{P}$.

\begin{proof}[Proof of Theorem \ref{thm_simultaneous polynomial common difference}]
Suppose $\mathcal{P}= \{P_1, \dots, P_{k}\} \subset \ZZ[t]$ is a collection of linearly independent polynomials with zero constant terms.  Fix \(\epsilon>0\) and let \(A \subseteq \FF_p\).  We apply Proposition~\ref{prop_arithmetic regularity decomposition} to $f=1_A$ with \(G=\mathbb F_p\).  This gives a decomposition
\[
    f=f_{\text{str}}+f_{\text{psr}}+f_{\text{sml}}
\]
with the properties described in Proposition~\ref{prop_arithmetic regularity decomposition}.  In particular,
\[
    \mathbb E f_{\mathrm{str}}
    =
    \mathbb E f
    =
    \alpha,
\]
the function \(f_{\mathrm{str}}\) is approximately invariant under translations by elements of \(B(\Gamma,\rho_1)\), the function \(f_{\mathrm{psr}}\) has Fourier coefficients bounded by \(\rho_2\), and \(f_{\mathrm{sml}}\) has \(L^2\)-norm at most \(\epsilon\).

Let $B=B(\Gamma,\rho_1)$. We define a set $B^*$ of Bohr-good parameters by
\[
    B^*:=\{d \in \mathbb F_p:\ P_i(d)\in B \text{ for all } i \in \{1, \dots, k\} \}.
\]
By Lemma~\ref{lma_equdistribution for poly. sequence},
\[
    \frac{1}{p} |B^*|
    =
    \left( \EE 1_B \right)^{k}
    +
    O_{\cP,|\Gamma|}\left(p^{-1/2} ( \log p)^{k|\Gamma|} \right).
\]
In particular, for \(p\) sufficiently large in terms of \(\mathcal P\), \(\Gamma\), and \(\rho_1\), Lemma~\ref{lma_B_size} implies that
\[
    \frac{1}{p}|B^*|
    \gtrsim
    \left(\mathbb E 1_B\right)^k
    \geq
    \rho_1^{k|\Gamma|}.
\]

We now examine the contributions of each of the three components $f_{\mathrm{str}}, f_{\mathrm{psr}}, f_{\mathrm{sml}}$ to \(I_{\mathcal P_\omega}\) for $\omega\in\{0,1\}^k$. 

\begin{enumerate}
    \item[1.] (\emph{Contribution of the structural part.})
Suppose \(d\in B^*\).  Then \(P_i(d)\in B\) for every \(i\).  Hence, whenever \(\omega_i=1\), the approximate invariance of \(f_{\mathrm{str}}\) gives that 
    \[
        f_{\mathrm{str}}(x+P_i(d))
        =
        f_{\mathrm{str}}(x)+O(\epsilon)
    \]
    uniformly in \(x\).  Therefore,
    \begin{align*}
        I_{\cP_{\omega}}(f_{\text{str}},\dots,f_{\text{str}})(d)
        &=
         \EE_{x \in \FF_p}
         f_{\text{str}}(x)
         \prod_{\substack{i \in \{1,\dots,k\}, \\ \omega_i=1}}
         f_{\text{str}}( x+ P_i(d))\\
        &=
         \EE_{x \in \FF_p}
         f_{\text{str}}(x)^{1+\sum_i\omega_i}
         +O_k(\epsilon).
    \end{align*}
    Since \(f_{\mathrm{str}}\) is nonnegative, Jensen's inequality gives
    \[
        \EE_{x \in \FF_p}
        f_{\text{str}}(x)^{1+\sum_i\omega_i}
        \geq
        \left(\EE_{x \in \FF_p} f_{\text{str}}(x)\right)^{1+\sum_i\omega_i}.
    \]
    Thus, for every \(d\in B^*\),
    \[
        I_{\cP_{\omega}}(f_{\text{str}},\dots,f_{\text{str}})(d)
        \geq
        \left(\EE_{x \in \FF_p} f_{\text{str}}(x)\right)^{1+\sum_i\omega_i}
        +O_k(\epsilon).
    \]

    \item[2.] (\emph{Control of the pseudorandom part.})
Consider any term in the expansion of
    \[
        I_{\mathcal P_\omega}(f,\dots,f)(d)
    \]
    in which at least one active slot is occupied by \(f_{\mathrm{psr}}\).  Here the active slots are the base slot \(0\) and the slots \(i\) for which \(\omega_i=1\).  Applying Lemma~\ref{lma_U^2-control for poly}, with the inactive slots filled by the constant function \(1\), gives a constant \(C=C(\mathcal P)>0\) such that
    \[
    \EE_{d \in \FF_p} |I_{\cP_{\omega}}(f_0,\dots,f_k)(d)|^2 \leq
     \|\widehat{f_{\text{psr}}}\|_{\ell^{\infty}(\FF_p)}
     +O_{\mathcal{P}}(p^{-C}).
    \]
    By Markov's inequality, 
    \begin{align*}
        \frac{|\{d: |I_{\cP_{\omega}}(f_0,\dots,f_k)(d)|> \epsilon\}|}{|\FF_p|}
        &= \frac{|\{d: |I_{\cP_{\omega}}(f_0,\dots,f_k)(d)|^2> \epsilon^2\}|}{|\FF_p|}\\
        &\leq
        \epsilon^{-2}
        \left(
            \|\widehat{f_{\text{psr}}}\|_{\ell^{\infty}(\FF_p)}
            +O_{\mathcal{P}}(p^{-C})
        \right)\\
        &\leq
        \epsilon^{-2}
        \left(
            \rho_2+O_{\mathcal{P}}(p^{-C})
        \right).
    \end{align*}
    Since there are only \(O_k(1)\) choices of \(\omega\) and only \(O_k(1)\) terms in each expansion, the union of all exceptional sets arising from pseudorandom terms has density at most
    \[
        O_k\!\left(
        \epsilon^{-2}
        \left(
            \rho_2+O_{\mathcal P}(p^{-C})
        \right)
        \right).
    \]

    \item[3.] (\emph{Control of the small part.})
Consider any term in the expansion of
    \[
        I_{\mathcal P_\omega}(f,\dots,f)(d)
    \]
    in which at least one active slot is occupied by \(f_{\mathrm{sml}}\).  By Cauchy--Schwarz and the \(1\)-boundedness of the remaining factors, we have that, for every \(d\in\mathbb F_p\), 
    \[
        |I_{\cP_{\omega}}(f_0,\dots,f_k)(d)|
        \leq
        \|f_{\text{sml}}\|_{L^2(\mathbb F_p)}
        \leq
        \epsilon.
    \]
\end{enumerate}

Combining these estimates, we see that we can choose \(d\in B^*\) outside all of the pseudorandom exceptional sets provided that
\begin{equation} \label{eq:condition for parameters}
    \frac{1}{p}|B^*|
    >
    O_k\!\left(
        \epsilon^{-2}
        \left(
            \rho_2+O_{\mathcal P}(p^{-C})
        \right)
    \right).
\end{equation}
For such a \(d\), all pseudorandom terms are \(O_k(\epsilon)\), all small terms are \(O_k(\epsilon)\), and so
\begin{align*}
    I_{\cP_{\omega}}(f,\dots,f)(d)
    &=
    I_{\cP_{\omega}}(f_{\text{str}},\dots,f_{\text{str}})(d)
    +O_k(\epsilon)\\
    &\geq
    \left(
        \EE_{x \in \FF_p} f_{\text{str}}(x)
    \right)^{1+\sum_i\omega_i}
    +O_k(\epsilon).
\end{align*}
Since \(\mathbb E f_{\mathrm{str}}=\mathbb E f=\alpha\), this becomes
\[
    I_{\cP_{\omega}}(f,\dots,f)(d)
    \geq
    \alpha^{1+\sum_i\omega_i}
    +O_k(\epsilon)
\]
for every nonzero \(\omega\in\{0,1\}^k\).

To ensure that \eqref{eq:condition for parameters} holds, recall that 
\[
    \frac{1}{p}|B^*|
    \gtrsim
    \rho_1^{k|\Gamma|}.
\]
Thus, \eqref{eq:condition for parameters} holds as long as the growth functions \(\eta_1,\eta_2\) are chosen so that \(\rho_2\) is sufficiently small compared with \(\epsilon^2\rho_1^{k|\Gamma|}\) and \(p\) is sufficiently large that the error \(O_{\mathcal P}(p^{-C})\) is negligible.  With this choice of parameters, we obtain an appropriate \(d\in B^*\). Moreover, since our lower bound for \(|B^*|\) gives \(|B^*|>1\) for \(p\) sufficiently large, we may take $d$ to be nonzero. 

Finally, the case \(\omega=(0,\dots,0)\) is immediate, since
\[
    I_{\mathcal P_\omega}(1_A,\dots,1_A)(d)
    =
    \mathbb E_{x\in\mathbb F_p}1_A(x)
    =
    \alpha
    =
    \alpha^{1+\sum_i\omega_i}.
\]
After replacing \(\epsilon\) by a sufficiently small constant multiple of the \(\epsilon\) appearing in the statement of the theorem, this proves Theorem~\ref{thm_simultaneous polynomial common difference}.
\end{proof}

\section{Proof of Theorem \ref{thm:k>2}} \label{sec:negative-ap}

In this section we prove Theorem~\ref{thm:k>2}.  The argument has two parts.  First, we construct a bounded function
\[
    f:\mathbb F_p^n\to[0,1]
\]
of density \(1/2+o_n(1)\) for which the simultaneous popular difference conclusion fails.  Second, we pass from this weighted construction to an actual set \(A\subseteq\mathbb F_p^n\) via a standard random sampling argument.

The weighted statement is the following proposition.

\begin{proposition} \label{prop:k>2}
For all sufficiently large primes $p$, there exists a constant $c > 0$ such that, for all sufficiently large $n$, there is a function $f:\FF_p^n\to[0,1]$ with $\EE f=1/2+o_n(1)$ and
        \[
        \max_{d \neq 0} \min \{ I_3(f)(d), I_3(f)(2d) \} \leq \frac{1}{8}-c.
        \]
\end{proposition}

Let us first explain why Proposition~\ref{prop:k>2} implies Theorem~\ref{thm:k>2}.  Suppose \(f:\mathbb F_p^n\to[0,1]\) is as in Proposition~\ref{prop:k>2}.  We form a random subset \(A\subseteq\mathbb F_p^n\) by including each point \(x\in\mathbb F_p^n\) in $A$ independently with probability \(f(x)\).  Then
\[
    \mathbb E\,1_A(x)=f(x).
\]
By the bounded difference inequality, with probability tending to \(1\) as \(n\to\infty\), we have
\[
    \mathbb E_{x\in\mathbb F_p^n}1_A(x)
    =
    \mathbb E_{x\in\mathbb F_p^n}f(x)+o_n(1)
    =
    \frac12+o_n(1)
\]
and, uniformly for all \(d\neq 0\),
\[
    I_3(1_A)(d)
    =
    I_3(f)(d)+o_n(1).
\]
Indeed, changing the value of \(1_A\) at a single point affects the quantity \(I_3(1_A)(d)\) by at most \(O(p^{-n})\) and there are only \(p^n-1\) nonzero choices of \(d\).  Thus, a union bound over all \(d\neq 0\), together with concentration, gives simultaneous approximation of all the progression counts.  Hence, allowing for the slight decrease in the constant $c$, Proposition~\ref{prop:k>2} implies the existence of a deterministic set \(A\) satisfying the conclusion of Theorem~\ref{thm:k>2}. 

It remains to prove Proposition~\ref{prop:k>2}.  The construction uses a quadratic factor.  The idea is to choose \(f\) of the form
\[
    f(x)=F(Q(x)),
\]
where \(Q(x)=x_1^2+\cdots+x_n^2\).  For such functions, the three-term progression count can be reduced to a two-variable average over \(\mathbb F_p^2\).  This reduction is the content of the next subsection.

\subsection{Quadratic Fourier analysis} \label{sec:quadratic}
    
Define the quadratic polynomial $Q: \FF_p^n \rightarrow \FF_p$ by
\[
Q(x):=\sum_{i=1}^n x_i^2 
\]
for $x=(x_1,\dots,x_n)\in\FF_p^n$. 
We will make use of a function \(f : \FF_p^n \rightarrow \CC\) which is \(1\)-bounded and measurable
with respect to the factor\footnote{The idea of construction here is motivated by the arithmetic regularity decomposition for the \(U^3(\FF_p^n)\)-norm, where the structural part is given by averaging the original function over the atoms of a quadratic factor. We refer the reader to \cite[Section 3]{G07} for a more comprehensive introduction to the relevant decomposition.}
generated by \(Q\), that is, there exists \(F:\FF_p\to\CC\) such that
\[
f(x)=F(Q(x)).
\]
for all $x\in\FF_p^n$. 

For \(x\in\mathbb F_p^n\) and fixed \(d \neq 0\), define
\[
t:=Q(x), \, \,
u:=2\langle x,d\rangle, \, \,
v:=Q(d).
\]
Then, for \(k \in \FF_p\),
\[
Q(x+kd)=t+ku+k^2v.
\]
Therefore,
\[
I_3(f)(d)
=
\mathbb E_{x\in\FF_p^n}
\prod_{k=0}^2 F(t+ku+k^2v).
\]
The next lemma shows that, for fixed \(d\neq 0\), the pair
\[
    (t,u)=\bigl(Q(x),2\langle x,d\rangle\bigr)
\]
is equidistributed on \(\mathbb F_p^2\), up to an error which is \(O(p^{-n/2})\).  Consequently, the above average over \(x\in\mathbb F_p^n\) can be replaced by a uniform average over \(t,u\in\mathbb F_p\).
That is, 
\[
I_3(f)(d)
=
\mathbb E_{t,u\in\mathbb F_p}
F(t)\,F(t+u+v)\,F(t+2u+4v)
\;+\;
O\!\left(p^{-n/2}\right).
\]
Moreover, we also have 
\[
\EE_{x \in \FF_p^n} f(x)
=
\EE_{t \in \FF_p} F(t)
+
O\!\left(p^{-n/2}\right).
\]

\begin{lemma}\label{lem:fixed-d-equidistribution}
Fix $d\in\mathbb F_p^n$ with $d\neq 0$.  
For $(t,u)\in\mathbb F_p^2$, define
\[
N_d(t,u)
:=\#\Bigl\{x\in\mathbb F_p^n:\; Q(x)=t,\; 2\langle x,d\rangle=u\Bigr\}.
\]
Then
\[
\bigl|N_d(t,u)-p^{n-2}\bigr|
\;\lesssim
p^{n/2}.
\]
Equivalently, for uniformly random $x\in\mathbb F_p^n$,
\[
\left|
\mathbb P(Q(x)=t,\;2\langle x,d\rangle=u)
-\frac1{p^2}
\right|
\;\lesssim
p^{-n/2}.
\]
\end{lemma}

\begin{proof}  
We use the orthogonality identity
\[
\mathbf 1_{\{0\}}(a)=\EE_{\lambda\in\mathbb F_p} e_p(\lambda a)
\]
to conclude that 
\[
N_d(t,u)
=\EE_{\alpha,\beta\in\mathbb F_p}
e_p(-\alpha t-\beta u)\,
S_d(\alpha,\beta),
\]
where
\[
S_d(\alpha,\beta)
:=\sum_{x\in\mathbb F_p^n}
e_p\!\bigl(\alpha Q(x)+2\beta\langle x,d\rangle\bigr).
\]
The term \((\alpha,\beta)=(0,0)\) contributes exactly
\[
\frac1{p^2}\sum_{x\in\mathbb F_p^n}1=p^{n-2}.
\]
If \(\alpha=0\) and \(\beta\neq 0\), then
\[
S_d(0,\beta)=\sum_{x} e_p(2\beta\langle x,d\rangle)=0,
\]
since \(d\neq 0\).

Now assume \(\alpha\neq 0\).  Completing the square gives
\[
\alpha Q(x)+2\beta\langle x,d\rangle
=\alpha Q\!\left(x+\frac{\beta}{\alpha}d\right)
-\frac{\beta^2}{\alpha}Q(d).
\]
Since translation does not alter the sum over \(x\),
\[
S_d(\alpha,\beta)
=e_p\!\left(-\frac{\beta^2}{\alpha}Q(d)\right)
\sum_{x\in\mathbb F_p^n} e_p(\alpha Q(x)).
\]
The quadratic Gauss sum factorizes coordinatewise, i.e.,
\[
\sum_{x\in\mathbb F_p^n} e_p(\alpha Q(x))
=\prod_{i=1}^n \sum_{z\in\mathbb F_p} e_p(\alpha z^2).
\]
Each one-dimensional Gauss sum has magnitude \(\sqrt p\), so
\begin{equation}\label{eq:Sd-gauss-bound}
|S_d(\alpha,\beta)|\lesssim p^{n/2}.
\end{equation}

Therefore, by \eqref{eq:Sd-gauss-bound},
\begin{align*}
\bigl|N_d(t,u)-p^{n-2}\bigr|
\le
\frac1{p^2}
\sum_{\substack{\alpha\neq 0,\\ \beta\in\mathbb F_p}}
|S_d(\alpha,\beta)| 
\lesssim
\frac1{p^2}\cdot (p-1)\cdot p \cdot p^{n/2}
\lesssim p^{n/2},
\end{align*}
as required.
\end{proof}

\subsection{Preliminary results}

The construction of \(F:\mathbb F_p\to[0,1]\) will be guided by a trigonometric polynomial on the torus.  The basic idea is to build a mean-zero function \(u\) on \(\mathbb T\) such that, for every \(x\), at least one of \(u(x)\) and \(u(4x)\) is bounded away from zero in the negative direction.  This will later be transferred to \(\mathbb F_p\) and encoded into the Fourier coefficients of a function \(G:\mathbb F_p\to\mathbb R\).

\begin{lemma} \label{lma_special function}
There exists a bounded continuous function \(u:\mathbb{T} \to\RR\)  such that
\begin{enumerate}
    \item \(u(-x)=u(x)\) for all \(x\in \mathbb{T}\),
    \item \(\int_{x\in \mathbb{T}} u(x)\, dx=0\),
    \item \(\max_{x\in \mathbb{T}} \min \{u(x), u(4x)\} < -0.1\).
\end{enumerate}
\end{lemma}

\begin{proof}
Identify \(\mathbb{T}:=\mathbb{R}/\mathbb{Z}\) with \([0,1)\).
Let
\[
B_0:=\Big[\frac1{16},\frac18\Big],\;\;
B_1 := \Big[1-\frac18,1-\frac1{16}\Big] = \Big[\frac78,\frac{15}{16}\Big],\;\;
B:=B_0\cup B_1
=\Big[\frac1{16},\frac18\Big]\cup\Big[\frac78,\frac{15}{16}\Big].
\]
Then \(B\) is symmetric in the sense that \(x\in B\) implies \(-x\in B\) modulo \(1\).
Moreover,
\[
4B_0=\Big[\frac14,\frac12\Big],\; \;
4B_1=\Big[\frac12,\frac34\Big]\; (\mathrm{mod}\ 1),
\]
so
\[
4B=\Big[\frac14,\frac34\Big].
\]
In particular,
\begin{equation}\label{eq:4B-disjoint-B}
4B\cap B=\emptyset.
\end{equation}

Fix a small parameter \(\delta\in(0,10^{-2})\), for instance, \(\delta=10^{-3}\).  Define a continuous function \(\varphi:\mathbb{T}\to[0,1]\) supported on \(B\) by
\[
\varphi(x)=
\begin{cases}
0, & x\notin \big[\frac1{16},\frac18\big]\cup\big[\frac78,\frac{15}{16}\big],\\[4pt]
\psi(x), & x\in \big[\frac1{16},\frac18\big],\\[4pt]
\psi(1-x), & x\in \big[\frac78,\frac{15}{16}\big],
\end{cases}
\]
where \(\psi:[\frac1{16},\frac18]\to[0,1]\) is the function
\[
\psi(x)=
\begin{cases}
\frac{x-\frac1{16}}{\delta}, & x\in\big[\frac1{16},\frac1{16}+\delta\big],\\[6pt]
1, & x\in\big[\frac1{16}+\delta,\frac18-\delta\big],\\[6pt]
\frac{\frac18-x}{\delta}, & x\in\big[\frac18-\delta,\frac18\big].
\end{cases}
\]
By construction, \(\varphi\) is continuous, bounded, and even.
Let $m:=\int_{\mathbb{T}} \varphi(x)\,dx$ and define 
\[
u(x):=\varphi(x)-m.
\]
Then \(u\) is bounded, continuous, and even. Moreover, it has zero mean, since
\[
\int_{\mathbb{T}} u(x)\,dx=\int_{\mathbb{T}}\varphi(x)\,dx-m=m-m=0.
\]

We now show that, for every \(x\in\mathbb T\), at least one of \(u(x)\) and \(u(4x)\) equals \(-m\).
If \(x\notin B\), then \(\varphi(x)=0\), so \(u(x)=-m\).  If \(x\in B\), then \(4x\in 4B\) and, since \(4B\cap B=\emptyset\) by \eqref{eq:4B-disjoint-B}, we have \(\varphi(4x)=0\), so \(u(4x)=-m\).  Therefore,
\[
    \min\{u(x),u(4x)\}\leq -m
\]
for every \(x\in\mathbb T\). 

Finally, \(m\) is bounded below by the measure of the region where \(\varphi\equiv 1\).  On each component of \(B\), the plateau \(\varphi\equiv 1\) has length
\[
    \Big(\frac18-\frac1{16}\Big)-2\delta
    =
    \frac1{16}-2\delta.
\]
Thus,
\begin{equation}\label{eq:m-lower-bound-special-function}
m\ge 2\Big(\frac1{16}-2\delta\Big)=\frac18-4\delta.
\end{equation}
Choosing \(\delta=10^{-3}\), \eqref{eq:m-lower-bound-special-function} gives \(m\ge 0.121>0.1\).  Hence,
\[
\max_{x\in\mathbb{T}} \min\{u(x),u(4x)\}
\leq
-m
< -0.1,
\]
as required.
\end{proof}

The previous lemma gives a continuous function.  We next approximate it by a finite cosine polynomial.  The evenness and mean-zero conditions ensure that the approximation can be taken to involve only nonconstant cosine terms.  The small perturbation at the end is only used to guarantee that all the coefficients are nonzero, which will be convenient for our final construction over \(\mathbb F_p\).

\begin{corollary}\label{cor_special function}
There exist an integer \(N\ge 1\) and nonzero real coefficients \(a_1,\dots,a_N\) such that the trigonometric polynomial
\[
P(x):=\sum_{n=1}^N a_n \cos(2\pi n x)
\]
satisfies
\[
    \displaystyle \max_{x\in\mathbb T}\min\{P(x),P(4x)\}< -0.05.
\]
\end{corollary}

\begin{proof}
Let \(u:\mathbb{T} \to\RR\) be the function constructed in Lemma~\ref{lma_special function}.  We have that \(u\) is even, \(\int_{\mathbb T}u=0\), and
\[
    \max_{x\in\mathbb T}\min\{u(x),u(4x)\}< -0.1.
\]
In particular,  
we may choose a constant \(m_0>0.05\) such that
\begin{equation}\label{eq:u-negative-margin-special-function}
    \min\{u(x),u(4x)\}\leq -m_0
\end{equation}
for all $x \in \mathbb{T}$. 

By the Stone--Weierstrass theorem, trigonometric polynomials are uniformly dense in \(C(\mathbb T)\), the set of continuous real-valued functions on \(\mathbb T\).  Hence, for any \(\varepsilon>0\), there exists a trigonometric polynomial
\[
T(x)=c_0+\sum_{n=1}^N \big( \alpha_n\cos(2\pi n x)+\beta_n\sin(2\pi n x)\big)
\]
such that
\[
    \|T-u\|_\infty<\varepsilon.
\]
Define its even part $T_{\mathrm{even}}(x)$ by 
\[
T_{\mathrm{even}}(x):=\frac{T(x)+T(-x)}{2}
= c_0+\sum_{n=1}^N \alpha_n\cos(2\pi n x).
\]
Since \(u\) is even,
\[
    \|T_{\mathrm{even}}-u\|_\infty\le \|T-u\|_\infty<\varepsilon.
\]
We now enforce a zero mean by subtracting the constant term to get
\[
Q(x):=T_{\mathrm{even}}(x)-\int_{\mathbb T}T_{\mathrm{even}}(t)\,dt
=\sum_{n=1}^N \alpha_n\cos(2\pi n x).
\]
Then \(Q\) is an even cosine polynomial with \(\int_{\mathbb T}Q=0\).  Since \(\int_{\mathbb T}u=0\),
\begin{equation}\label{eq:Q-u-uniform-approximation}
\|Q-u\|_\infty \le \|T_{\mathrm{even}}-u\|_\infty
+\left|\int_{\mathbb T}(T_{\mathrm{even}}-u)\right|
< \varepsilon+\varepsilon=2\varepsilon.
\end{equation}
Choose \(\varepsilon>0\) sufficiently small that
\[
    2\varepsilon<\frac{m_0-0.05}{2}.
\]
Let \(\eta:=\|Q-u\|_\infty\).  Then, by \eqref{eq:Q-u-uniform-approximation},
\(\eta<\frac{m_0-0.05}{2}\).  Therefore, for every \(x \in \mathbb{T}\), by
\eqref{eq:u-negative-margin-special-function},
\[
\min\{Q(x),Q(4x)\}
\le
\min\{u(x),u(4x)\}+\eta
\le
-m_0+\eta
< -0.05.
\]
Thus,
\begin{equation}\label{eq:Q-negative-margin-special-function}
\max_{x\in\mathbb T}\min\{Q(x),Q(4x)\}< -0.05.
\end{equation}

Finally, let
\[
    S:=\{1\le n\le N:\ \alpha_n=0\}.
\]
If \(S=\emptyset\), we are done.  Otherwise, for a parameter \(\delta'>0\), define 
\[
R(x):=\delta'\sum_{n\in S}\cos(2\pi n x)
\]
and
\[
P(x):=Q(x)+R(x).
\]
Then
\[
    P(x)=\sum_{n=1}^N a_n\cos(2\pi n x),
\]
where \(a_n=\alpha_n\) for \(n\notin S\) and \(a_n=\delta'\) for \(n\in S\).  Hence, \(a_n\neq 0\) for all \(1\le n\le N\).

Moreover,
\[
    \|R\|_\infty\le \delta'|S|.
\]
Choose \(\delta'>0\) so small that \(\delta'|S|<\gamma\), where
\[
\gamma:=\frac12\Big(-0.05 - \max_{x\in\mathbb T}\min\{Q(x),Q(4x)\}\Big)>0.
\]
Here \(\gamma>0\) by \eqref{eq:Q-negative-margin-special-function}.  Then, for every \(x \in \mathbb{T}\),
\[
P(x)\le Q(x)+\gamma, \, \, P(4x)\le Q(4x)+\gamma,
\]
so 
\[
\min\{P(x),P(4x)\}
\le
\min\{Q(x),Q(4x)\}+\gamma.
\]
Taking the maximum over all \(x \in \mathbb{T}\) gives
\[
\max_{x\in\mathbb T}\min\{P(x),P(4x)\}< -0.05,
\]
as required.
\end{proof}

\subsection{Proof of Proposition \ref{prop:k>2}}

We now complete the construction of the weighted counterexample.  The previous corollary gives a trigonometric polynomial \(P\) on the torus satisfying
\[
    \max_{x \in \mathbb{T}}\min\{P(x),P(4x)\}< -0.05.
\]
Our goal now is to realize this trigonometric polynomial as the nonzero-frequency part of the three-term progression count associated with a function \(F:\mathbb F_p\to[0,1]\).

For \(F:\FF_p\to\RR_{\ge 0}\), define
\[
I_3(k;v)
:=
\EE_{t,u\in\FF_p}
F(t)\,F(t+ku+k^2v)\,F(t+2ku+4k^2v).
\]
By the reduction in Section~\ref{sec:quadratic}, if \(f(x)=F(Q(x))\), the quantity \(I_3(f)(d)\) is approximated by \(I_3(1;Q(d))\), while \(I_3(f)(2d)\) is approximated by \(I_3(2;Q(d))\).  Thus, it is enough to construct \(F:\FF_p\to[0,1]\) with \(\mathbb E F=1/2\) such that
\[
\max_{v \in \FF_p} \min \{I_3(1;v),I_3(2;v)\} \leq (1/2)^3 -c
\]
for some absolute constant \(c>0\).

To better understand \(I_3(k;v)\), we make a change of variables.  For \(k\neq 0\), set \(w:=ku+k^2v\).  Then \(w\) is uniform in \(\FF_p\) and if we define
\[
J_3(s):=\EE_{t,w\in\FF_p} F(t)\,F(t+w)\,F(t+2w+s),
\]
then
\[
    I_3(k;v)=J_3(2k^2v).
\]
Indeed, after setting \(w=ku+k^2v\), the third argument in $I_3(k;v)$ becomes
\[
    t+2ku+4k^2v=t+2w+2k^2v,
\]
which is the third argument in $J_3(s)$ when \(s=2k^2v\). 

Therefore, our goal becomes to show that 
\begin{equation} \label{Condition for counterexample}
    \max_{s \in \FF_p} \min \{J_3(s),J_3(4s)\} \leq (1/2)^3 -c.
\end{equation}
To this end, we rewrite \(J_3(s)\) in Fourier terms, as 
\begin{align*}
    J_3(s)
    &=\EE_{t,w\in\FF_p}
    \sum_{\xi_1,\xi_2,\xi_3\in\FF_p}
    \widehat F(\xi_1) \widehat F(\xi_2) \widehat F(\xi_3)
    e_p(\xi_1t+\xi_2(t+w)+\xi_3(t+2w+s))\\
    &=\sum_{\xi_1,\xi_2,\xi_3 \in\FF_p}
    \widehat F(\xi_1) \widehat F(\xi_2) \widehat F(\xi_3)
    e_p(\xi_3s)\cdot
    \EE_{t,w\in\FF_p}
    e_p(\xi_1t+\xi_2(t+w)+\xi_3(t+2w))\\
    &=\sum_{\xi_1,\xi_2,\xi_3\in\FF_p}
    \widehat F(\xi_1) \widehat F(\xi_2) \widehat F(\xi_3)
    e_p(\xi_3s)\cdot
    \mathbf 1_{\xi_1+\xi_2+\xi_3=0}
    \mathbf 1_{\xi_2+2\xi_3=0}\\
    &=\sum_{\xi \in \FF_p}
    \widehat F(\xi)\widehat F(-2\xi)\widehat F(\xi)e_p(\xi s)\\
    &=\sum_{\xi \in \FF_p}
    \widehat F(\xi)^2 \widehat F(-2\xi)e_p(\xi s),
\end{align*}
where we used the parametrization \(\xi_1=\xi_3=\xi\) and \(\xi_2=-2\xi\).

We now construct the required function \(F\).  The following lemma says that any finite cosine polynomial with nonzero coefficients can be realized as the nonzero-frequency contribution
\[
    \sum_{\xi\neq 0}
    \widehat G(\xi)^2\widehat G(-2\xi)e_p(\xi s)
\]
for a suitable real-valued mean-zero function \(G:\mathbb F_p\to\mathbb R\).

\begin{lemma}\label{lma_assign value}
Let $N\ge 1$ and let
\[
P_p(s):=\sum_{n=1}^N a_n \cos\!\Big(\frac{2\pi n s}{p}\Big)
\]
be a trigonometric polynomial on $\mathbb F_p$ (viewing $s\in\mathbb F_p$ as an integer in $\{0,1,\dots,p-1\}$) with nonzero real coefficients $a_n$.
Provided $p>10N$, there exists a real-valued function $G:\mathbb F_p\to\mathbb R$ with $\EE G=0$ such that,
for every $s\in\mathbb F_p$,
\begin{equation}\label{eq:target-identity}
\sum_{\xi\in\mathbb F_p \setminus \{0\}}\widehat G(\xi)^2 \widehat G(-2\xi) e_p(\xi s)
\;=\;P_p(s).
\end{equation}
\end{lemma}

Assuming Lemma~\ref{lma_assign value} for the moment, we finish the proof of Proposition~\ref{prop:k>2}.  By Corollary~\ref{cor_special function}, we may choose \(N\) and nonzero coefficients \(a_1,\dots,a_N\) such that
\[
    P(x)=\sum_{n=1}^N a_n\cos(2\pi n x)
\]
satisfies
\[
    \max_{x\in\mathbb T}\min\{P(x),P(4x)\}< -0.05.
\]
For \(p>10N\), apply Lemma~\ref{lma_assign value} to the finite-field polynomial
\[
    P_p(s)=\sum_{n=1}^N a_n\cos\!\Big(\frac{2\pi n s}{p}\Big)
\]
to obtain a real-valued function \(G:\mathbb F_p\to\mathbb R\) with \(\mathbb E G=0\) such that
\[
\max_{s\in\mathbb F_p}
\min\left\{
\sum_{\xi\in\mathbb F_p \setminus \{0\}}\widehat G(\xi)^2\widehat G(-2\xi)e_p(\xi s),
\sum_{\xi\in\mathbb F_p \setminus \{0\}}\widehat G(\xi)^2\widehat G(-2\xi)e_p(4\xi s)
\right\}
< -0.05.
\]

Now set
\[
F(x)=\frac{1}{2}+\frac{G(x)}{2\|G\|_{\infty}}.
\]
Then \(0\leq F\leq 1\) and \(\mathbb E F=1/2\).  
Thus, the zero frequency of \(F\) is \(1/2\), while, for \(\xi\neq 0\),
\[
    \widehat F(\xi)=\frac{\widehat G(\xi)}{2\|G\|_\infty}.
\]
Hence,
\[
    J_3(s)
    =
    \frac18
    +
    \frac{1}{(2\|G\|_\infty)^3}
    \sum_{\xi\neq 0}
    \widehat G(\xi)^2\widehat G(-2\xi)e_p(\xi s).
\]
Therefore,
\[
    \max_{s\in\mathbb F_p}\min\{J_3(s),J_3(4s)\}
    \leq
    \frac18-\frac{0.05}{(2\|G\|_\infty)^3}.
\]
Thus, \eqref{Condition for counterexample} holds with
\[
    c=\frac{0.05}{(2\|G\|_\infty)^3}.
\]

Finally, define
\[
    f(x):=F(Q(x))
\]
for $x\in\mathbb F_p^n$. 
The equidistribution estimate Lemma~\ref{lem:fixed-d-equidistribution} gives
\[
    \mathbb E_{x\in\mathbb F_p^n}f(x)
    =
    \mathbb E_{t\in\mathbb F_p}F(t)+O(p^{-n/2})
    =
    \frac12+o_n(1)
\]
and, uniformly for every \(d\neq 0\),
\[
    I_3(f)(d)
    =
    I_3(1;Q(d))+O(p^{-n/2})
\]
and 
\[
    I_3(f)(2d)
    =
    I_3(2;Q(d))+O(p^{-n/2}).
\]
Therefore, since $I_3(k;Q(d))=J_3(2k^2Q(d))$ for $k = 1, 2$, 
\[
    \max_{d\neq 0}
    \min\{I_3(f)(d),I_3(f)(2d)\}
    \leq
    \frac18-c+o_n(1).
\]
For \(n\) sufficiently large, after replacing \(c\) by \(c/2\), this proves Proposition~\ref{prop:k>2}.

It remains only to prove Lemma~\ref{lma_assign value}.

\begin{proof}[Proof of Lemma \ref{lma_assign value}]
For any \(G:\mathbb F_p\to\mathbb C\), define
\[
H(s):=\sum_{\xi\in\mathbb F_p}\widehat G(\xi)^2\,\widehat G(-2\xi)\,e_p(\xi s).
\]
Then the Fourier coefficient of \(H\) at frequency \(\xi\) is
\[
\widehat H(\xi)=\widehat G(\xi)^2\,\widehat G(-2\xi).
\]
On the other hand,
\[
P_p(s)=\sum_{n=1}^N \frac{a_n}{2}\big(e_p(ns)+e_p(-ns)\big),
\]
so \(\widehat{P_p}(\pm n)=a_n/2\) for \(1\le n\le N\) and \(\widehat {P_p}(\xi)=0\) otherwise.

Thus, it suffices to choose \(\widehat G\) such that
\begin{equation}\label{eq:coeff-eqns}
\widehat G(\xi)^2\,\widehat G(-2\xi)=\widehat{P_p}(\xi)
\end{equation}
for all $\xi\in\mathbb F_p$. 
We impose
\begin{equation}\label{eq:supp}
\widehat G(\xi)=0 \;\; \text{unless}\;\; \xi\in\{\pm 1,\dots,\pm 2N\}.
\end{equation}
Since \(p>10N\), the integers \(\pm 1,\dots,\pm 4N\) represent distinct nonzero elements of \(\mathbb F_p\).  Hence, if \(N<|\xi|\le 2N\), then \(|-2\xi|>2N\), so by \eqref{eq:supp} we have \(\widehat G(-2\xi)=0\) and, therefore, \(\widehat H(\xi)=0\).  If \(|\xi|>2N\), then \(\widehat G(\xi)=0\), so again \(\widehat H(\xi)=0\).  Consequently, \(\widehat H(\xi)=0\) for all \(|\xi|>N\), so \eqref{eq:coeff-eqns} reduces to the requirement that 
\[
\widehat H(\pm n)=\frac{a_n}{2} 
\]
for all $1\le n\le N$. 

Write
\[
\widehat G(\pm k)=x_k\in\mathbb R 
\]
for $1\le k\le 2N$ 
and \(\widehat G(\xi)=0\) otherwise.  Then, for \(1\le n\le N\),
\[
\widehat H(n)=\widehat H(-n)=x_n^2\,x_{2n}.
\]
Thus, it remains to solve the real system
\begin{equation}\label{eq:chain-system}
x_n^2\,x_{2n}=b_n
\end{equation}
for $1\le n\le N$, where $b_n:=\frac{a_n}{2}\neq 0$. 

For each odd integer \(m\le N\), consider the doubling chain
\[
m, 2m, 4m, \dots, 2^k m \le N < 2^{k+1}m \le 2N,
\]
where \(k\ge 0\) is maximal with \(2^k m\le N\).  

Fix such an odd \(m\).  Choose an arbitrary nonzero real number \(x_{2^{k+1}m}\) with sign satisfying
\[
\mathrm{sgn}(x_{2^{k+1}m})=\mathrm{sgn}(b_{2^k m}),
\]
so that
\[
    \frac{b_{2^k m}}{x_{2^{k+1}m}}>0.
\]
Now define recursively, for \(j=k,k-1,\dots,0\),
\[
x_{2^j m}\ :=\ \sqrt{\frac{b_{2^j m}}{x_{2^{j+1}m}}}.
\]
Then, by construction,
\[
x_{2^j m}^2\,x_{2^{j+1}m}=b_{2^j m}
\]
for $j=0,1,\dots,k$, which is exactly \eqref{eq:chain-system} for all \(n\) in this chain.

Perform this construction independently for each odd \(m\le N\).  Since every \(1\le n\le N\) can be written uniquely as \(n=2^j m\) for some odd \(m\), this defines \(x_n\) and \(x_{2n}\) for all \(1\le n\le N\). 

Finally, with \(\widehat G\) defined as above, the function \(G\) is real-valued because
\[
    \widehat G(\xi)=\widehat G(-\xi)=\overline{\widehat G(\xi)}.
\]
Moreover, \(\widehat G(0)=0\), so \(\mathbb E G=0\).  By construction, \(\widehat H(\xi)=\widehat {P_p}(\xi)\) for every \(\xi\in\mathbb F_p\).  Hence, \(H(s)=P_p(s)\) for all \(s\in\mathbb F_p\), which is exactly \eqref{eq:target-identity}.
\end{proof}

\section{Concluding remarks}

\subsection*{Quantitative bounds}

Over $\mathbb{F}_p$, Green's result states that if $p \ge p_0(\varepsilon)$, then every $A \subseteq \mathbb{F}_p$ of density $\alpha$ contains a nonzero $d$ such that the density of three-term arithmetic progressions in $A$ with common difference $d$ is at least $\alpha^3 - \varepsilon$. Like many results proved using regularity methods, Green's proof of his popular difference theorem gives a tower-type bound on $p_0(\varepsilon)$. A surprising result of Fox, Pham, and Zhao~\cite{FP21, FPZ23} shows that such a bound is necessary, suggesting that regularity is necessary in this context. It would be interesting to decide whether the same is true of our Theorem~\ref{thm_simultaneous polynomial common difference}. We are not sure which way the truth should lie, as it is plausible that regularity could be dispensed with in this setting.

\subsection*{Rational functions}

For three-term progressions, that is, when $k = 2$, Theorem \ref{thm_simultaneous polynomial common difference} can be extended to rational function progressions, assuming that the two rational functions $P_1, P_2$ that appear and the constant function $1$ are linearly independent. A result of Hong and Lim~\cite{HL25} (see also~\cite{H25, L25}), answering a question of Bourgain and Chang~\cite{BC17}, gives the analogue of Peluse's result, Theorem~\ref{thm_polynomial Szemeredi}, in this setting. The analogue of Lemma~\ref{lma_equdistribution for poly. sequence}, saying that independent rational functions equidistribute over a low-rank Bohr set, requires exponential sum estimates beyond the Weil bound, but the proof is otherwise the same.

\subsection*{Further variations}

In Theorem~\ref{thm:k>2}, we showed that there are sets $A \subseteq \mathbb{F}_p^n$ for which there is no nonzero $d$ such that both $d$ and $2d$ are popular differences. We do not know the answer for the analogous question with $d$ and $d^2$. Moreover, while it is true, by an application of Green's regularity lemma~\cite{G05}, that there are simultaneous popular differences $d$ for $x, x+d$ and $x, x+d, x+2d$, we do not know if there are for $x, x+d, x+2d$ and $x, x+d, x+2d, x+3d$, that is, for both three- and four-term progressions. We do not even know the answer for $x, x+d, x+2d$ and $x, x+d, x+3d$.


\bibliographystyle{alpha}
\bibliography{bibliography}

\end{document}